\documentclass[a4paper,11pt,DIV=11,oneside,headings=small,abstracton]{scrartcl}
\usepackage{amssymb,amsmath,amsfonts,amsthm}
\usepackage[english]{babel}
\usepackage[utf8]{inputenc}
\usepackage{enumerate,expdlist}
\usepackage{tikz,float} 
\usetikzlibrary{patterns}
\usepackage[noblocks]{authblk}

\usepackage{microtype}
\usepackage{ellipsis}
\usepackage{typearea}
\newtheorem{theorem}{Theorem}[section]
\newtheorem{lemma}[theorem]{Lemma}
\newtheorem{proposition}[theorem]{Proposition}
\newtheorem{corollary}[theorem]{Corollary}
\theoremstyle{definition}
\newtheorem{definition}[theorem]{Definition}
\theoremstyle{remark}
\newtheorem{remark}[theorem]{Remark}
\renewcommand{\epsilon}{\varepsilon}
\renewcommand{\index}{\mathcal{I}}
\newcommand{\euler}{\mathrm{e}}
\newcommand{\rout}{R}
\newcommand{\sfuc}{\mathrm{sfuc}}
\newcommand{\RR}{\mathbb{R}}
\newcommand{\CC}{\mathbb{C}}
\newcommand{\NN}{\mathbb{N}}
\newcommand{\ZZ}{\mathbb{Z}}
\newcommand{\Ca}{\mathrm{B}}
\newcommand{\Cb}{\mathrm{A}}
\newcommand{\Da}{D_{\mathrm{B}}}
\newcommand{\Db}{D_{\mathrm{A}}}
\newcommand{\Na}{N_{\mathrm{B}}}
\newcommand{\Nb}{N_{\mathrm{A}}}
\DeclareMathOperator{\funs}{s}
\begin{document}
\title{\Large Scale-free and quantitative unique continuation for infinite dimensional spectral subspaces of Schr\"odinger operators}
\author{Matthias T\"aufer}
\affil{Technische Universit\"at Dortmund, Fakult\"at f\"ur Mathematik, Germany}
\author{Martin Tautenhahn}
\affil{Technische Universit\"at Chemnitz, Fakult\"at f\"ur Mathematik, Germany}
\date{\vspace{-3em}}
\maketitle

\begin{abstract}
We prove a quantitative unique continuation principle for infinite dimensional spectral subspaces of Schr\"odinger operators.
Let $\Lambda_L = (-L/2,\allowbreak L/2)^d$ and $H_L = -\Delta_L + V_L$ be a Schr\"odinger operator on $L^2 (\Lambda_L)$ with a bounded potential $V_L : \Lambda_L \to \RR^d$ and Dirichlet, Neumann, or periodic boundary conditions.
Our main result is of the type
  \[
   \int_{\Lambda_L} \lvert \phi \rvert^2 \leq C_\sfuc \int_{W_\delta (L)} \lvert \phi \rvert^2,
  \]
  where $\phi$ is an infinite complex linear combination of eigenfunctions of $H_L$ with exponentially decaying coefficients, $W_\delta (L)$ is some union of equidistributed $\delta$-balls in $\Lambda_L$ and $C_{\sfuc} > 0$ an $L$-independent constant.
  The exponential decay condition on $\phi$ can alternatively be formulated as an exponential decay condition of the map $\lambda \mapsto \lVert \chi_{[\lambda , \infty)} (H_L) \phi \rVert^2$.
  The novelty is that at the same time we allow the function $\phi$ to be from an infinite dimensional spectral subspace and keep an explicit control over the constant $C_{\sfuc}$ in terms of the parameters.
  Moreover, we show that a similar result cannot hold under a polynomial decay condition.

\end{abstract}
\section{Introduction}
Starting with the pioneering work \cite{Carleman-39}, there has been plenty of research concerning unique continuation properties for elliptic operators $L$ with non-analytic coefficients. That is, if the solution $u$ of $L u = 0$ in $\Omega \subset \RR^d$ vanishes in a non-empty open set $\omega \subset \Omega$, then $u$ will be identically zero, see e.g.\ \cite{Hoermander-89} and the references therein.
More than this, there are several quantitative formulations of unique continuation which proved to be useful in a variety of applications, see e.g.\ \cite{BourgainK-05,RousseauL-12,BourgainK-13,Rojas-MolinaV-13,NakicTTV}. For instance, Bourgain and Kenig \cite{BourgainK-05} showed that if $\Delta u = V u$ in $\RR^d$, $u (0) = 1$ and $u,V \in L^\infty(\RR^d)$ then for all $x \in \RR^d$ with $\lvert x \rvert > 1$ we have
 \begin{equation} \label{eq:BK05}
  \max_{\lvert y-x \rvert \leq 1} \lvert u (y) \rvert > c \cdot \exp \left( - c' (\log \lvert x \rvert) \lvert x \rvert^{4/3}\right) .
 \end{equation}
This quantitative formulation has been crucial for the proof of
Anderson localization for the continuum Anderson model with Bernoulli-distributed coupling constants. An $L^2$-variant of Ineq.~\eqref{eq:BK05} has been shown in \cite{BourgainK-13} in order to study the density of states of Schr\"odinger operators.
A similar quantitative formulation is an estimate of the type
\begin{equation} \label{eq:qucp_intro}
  \lVert u \rVert_{L^2(\omega)}^2  \geq C\lVert u \rVert_{L^2(\Omega)}^2 ,
\end{equation}
where $u$ is in the range of some spectral projector of a Schr\"odinger operator with potential $V$, and $C$ is some positive constant depending on the geometry of $\omega$ and the potential $V$.
Such quantitative unique continuation principles have been applied to control theory for the heat equation and spectral theory of random Schr\"odinger operators, see e.g.\ the recent \cite{TaeuferTV-16} and the references therein. Let us emphasize that the dependence of $C$ on the geometry of $\omega$ turned out to be important for some of these applications. To be more specific, let $\Omega \subset \RR^d$ be a finite, open and non-empty connected set, $W \in L^\infty (\Omega)$ and
\[
 H_\Omega = -\Delta + W \quad \text{on} \quad L^2 (\Omega)
\]
with Dirichlet boundary conditions
Then Ineq.~\eqref{eq:qucp_intro} has been obtained in \cite{RousseauL-12} in the case $W \equiv 0$, $\omega \subset \Omega$ open and non-empty, $u$ a (finite or infinite) linear combinations of eigenfunctions of $H_\Omega$. However, the dependence of $C$ on the geometry of $\omega$ is not known.
In \cite{Rojas-MolinaV-13} Ineq.~\eqref{eq:qucp_intro} is proven for $\Omega = (-L/2 , L/2)^d$, $\omega$ an equidistributed arrangement of $\delta$-balls, $u \in W^{2,2} (\Omega)$ satisfying $\lvert \Delta u \rvert \leq \lvert W u \rvert$, and with $$C = \delta^{N (1 + \lVert W \rVert_\infty^{2/3})}$$ where $N > 0$ depends only on the dimension. For the application to random Schr\"odinger operators it is crucial that the result is scale-free, i.e.\ $C$ is independent of $L$. In \cite{Rojas-MolinaV-13} the question was raised whether a similar estimate holds for finite linear combinations of eigenfunctions $u \in \operatorname{Ran} \chi_{(-\infty , b]} (H_\Omega)$. A partial answer to this question was given in \cite{Klein-13}. The full answer has been announced in \cite{NakicTTV-15}, and full proofs have been given in \cite{NakicTTV}. There, the constant
\[
C = \delta^{N (1 + \lVert W \rVert_\infty^{2/3} + \sqrt{\lvert b \rvert})}
\]
is derived.
Let us emphasize that this was the missing step to study localization for random Schr\"odinger operators with non-linear dependence on the
random parameters.

The aim of this note is to extend the main result of \cite{NakicTTV} to the natural setting of infinite dimensional spectral subspaces.
For this purpose, we first extend the strategy of \cite{NakicTTV} to prove Ineq.~\eqref{eq:qucp_intro} for infinite linear combinations of eigenfunctions with exponentially decaying coefficients, cf.\ Theorem~\ref{thm:result1}.
In a second step we show that Ineq.~\eqref{eq:qucp_intro} holds if $\lVert \chi_{[\lambda , \infty)} (H_\Omega) \phi \rVert^2$ decays exponentially in $\lambda$, cf.\ Theorem~\ref{thm:result2}.
In order not to lose the explicit control over the constant $C_\sfuc$, in particular its $L$-independence, this step requires a detailed analysis using precise knowledge of the $\Delta$-eigenvalues and eigenfunctions, cf.\ Lemma~\ref{lemma:B->A}.
While the proofs are given in Section~\ref{sec:proofs}, we will show in Section~\ref{sec:optimality} that our results are optimal in the sense that they cannot hold under polynomial decay conditions.

\section{Notation and main results}
\label{sec:results}
Let $d \in \NN$.
For $L,r > 0$ we denote by $\Lambda_L = (-L/2 , L/2)^d \subset \RR^d$ the $d$-dimensional cube with side length $L$ and by $B(x , r)$ the ball with center $x$ and radius $r$ with respect to the Euclidean norm.
The Laplace operator on $L^2 (\Lambda_L)$ with Dirichlet, Neumann or periodic boundary conditions is denoted by $\Delta_L$.
For $\Omega \subset \RR^d$ open and $\psi \in L^2 (\Omega)$ we denote by $\lVert \psi \rVert = \lVert \psi \rVert_\Omega = \lVert \psi \rVert_{L^2(\Omega)}$ the usual $L^2$-norm of $\psi$.
If $\Gamma \subset \Omega$ we use the notation $\lVert \chi_\Gamma \psi \rVert_\Omega = \lVert \psi \rVert_\Gamma = \lVert \psi \rVert_{L^2(\Gamma)} $.
Moreover, for a measurable and bounded $V : \RR^d \to \RR$ we denote by $V_L : \Lambda_L \to \RR$ its restriction to $\Lambda_L$ given by $V_L (x)  = V (x)$ for $x \in \Lambda_L$, and by
\[
H_L = -\Delta_L + V_L \quad \text{on} \quad L^2 (\Lambda_L)
\]
the corresponding Schr\"odinger operator.
We will also write $V = V_+ - V_-$ for the decomposition into positive and negative part and $\lVert V \rVert_\infty$ for the $L^\infty$-norm of $V$.
The operator $H_L$ is lower semibounded and self-adjoint with lower bound $-\lVert V_- \rVert_\infty$ and purely discrete spectrum.
\begin{definition}
 Let $G > 0$ and $\delta > 0$. We say that a sequence $Z = (z_j)_{j \in (G\ZZ)^d} \subset \RR^d$ is \emph{$(G,\delta)$-equidistributed}, if
 \[
  \forall j \in (G \ZZ)^d \colon \quad  B(z_j , \delta) \subset \Lambda_G + j = \{x + j \in \RR^d : x \in \Lambda_L \}.
\]
Corresponding to a $(G,\delta)$-equidistributed sequence $Z$ we define for $L \in G \NN$ the set
\[
W_\delta (L) = \bigcup_{j \in (G \ZZ)^d } B(z_j , \delta) \cap \Lambda_L ,
\]
where we suppressed the dependence of $W_\delta (L)$ on $G$ and on the choice of $Z$.
\end{definition}
\begin{theorem}\label{thm:result2}
There is $\Nb = \Nb(d)>0$ such that for
all $\kappa > 0$,
all $G \in (0, \kappa/ \allowbreak ( 18 \euler \sqrt{d}))$,
all $\delta \in (0,G/2)$,
all $(G,\delta)$-equidistributed sequences $Z$,
all measurable and bounded $V: \RR^d \to \RR$,
all $L \in G \NN$,
all $\Db \geq 1$
and all $\phi \in L^2 (\Lambda_L)$ satisfying
\begin{equation}\label{eq:Assumption_B}
\text{for all} \
 \lambda \in [- \lVert V_- \rVert_\infty, \infty) : \quad \left\lVert \chi_{[\lambda , \infty)} (H_L) \phi \right\rVert_{\Lambda_L}^2
 \leq
 \Db \euler^{- \kappa \sqrt{\lambda + \lVert V_- \rVert_\infty}} \lVert \phi \rVert_{\Lambda_L}^2 ,
\end{equation}
we have
\begin{equation} \label{eq:result1}
\lVert \phi \rVert_{W_\delta (L)}^2
\geq C_\sfuc^{\Cb} \lVert \phi \rVert_{\Lambda_L}^2
\end{equation}
where
\[
 C_\sfuc^{\Cb} = C_{\sfuc}^{\Cb} (d, \delta ,  \Db  , \lVert V \rVert_\infty) := \left( \frac{\delta}{G} \right)^{\Nb \bigl(1 + G^{4/3} \lVert V \rVert_\infty^{2/3}  +  \ln \Db  + G / ( \kappa - G 18 \euler \sqrt{d}) \bigr )} .
 \quad
\]
\end{theorem}
For every measurable and bounded $V:\RR^d \to \RR$ and every $L \in \NN$ we denote the eigenvalues of the corresponding operator $H_L$ by $E_k$, $k \in \NN$, enumerated in increasing order and counting multiplicities, and fix a corresponding sequence $\psi_k$, $k \in \NN$, of normalized eigenfunctions. Note that we suppress the dependence of $E_k$ and $\psi_k$ on $V$ and $L$. For $\phi \in L^2 (\Lambda_L)$ we set $\alpha_k = \langle \psi_k , \phi \rangle$, whence
\[
 \phi = \sum_{k \in \NN} \alpha_k \psi_k .
\]
\begin{theorem}\label{thm:result1}
There is $\Na = \Na (d) > 0$ such that for
all $\kappa > 0$,
all $G \in (0, \kappa/ ( 18 \euler \sqrt{d})]$,
all $\delta \in (0,G/2)$,
all $(G,\delta)$-equidistributed sequences $Z$,
all measurable and bounded $V: \RR^d \to \RR$, all $L \in G \NN$,
all $\Da \geq 1$ and all $\phi \in L^2 (\Lambda_L)$ satisfying
\begin{equation}\label{eq:Assumption_A}
 \sum_{k\in \NN} \exp \left( \kappa \sqrt{\max\{0,E_k\}} \right) \lvert \alpha_k \rvert^2 \leq \Da \sum_{k \in \NN} \lvert \alpha_k \rvert^2 ,
\end{equation}
we have
\begin{equation*}
\lVert \phi \rVert_{W_\delta (L)}^2
\geq C_\sfuc^{\Ca} \lVert \phi \rVert_{\Lambda_L}^2
\end{equation*}
where
\[
 C_\sfuc^{\Ca} = C_{\sfuc}^{\Ca} (d, G, \delta ,  \Da  , \lVert V \rVert_\infty ) := \left( \frac{\delta}{G} \right)^{\Na \bigl(1 + G^{4/3}\lVert V \rVert_\infty^{2/3} +  \ln \Da \bigr)}.
\]
\end{theorem}
A special case of Theorem~\ref{thm:result2} and~\ref{thm:result1} is $\phi \in \mathrm{Ran} (\chi_{(-\infty,b]}(H_L))$ for some $b \geq - \lVert V_- \rVert_\infty$.
Let us assume $G = 1$ for convenience.
In this case Inequality~\eqref{eq:Assumption_A} holds with $\kappa = 18 \euler \sqrt{d}$ and $\Da = \exp(18 \euler \sqrt{d} \sqrt{b + \lVert V_- \rVert_\infty})$. Inequality~\eqref{eq:Assumption_B} holds, e.g., with $\kappa = 18 \euler \sqrt{d} + 1$ and $\Db = \exp((18 \sqrt{d} \euler + 1) \sqrt{b + \lVert V_- \rVert_\infty})$.
Hence, the constants $C_\sfuc^{\Ca, \Cb}$ in Theorem~\ref{thm:result1} and \ref{thm:result2} can  be estimated as
\[
 C_\sfuc^{\Ca} \geq \delta^{\tilde \Na \bigl(1 + \lVert V \rVert_\infty^{2/3} + \sqrt{\lvert b \rvert} \bigr)},
 \quad
 \text{and}
 \quad
 C_\sfuc^{\Cb} \geq \delta^{\tilde \Nb \bigl(1 + \lVert V \rVert_\infty^{2/3} + \sqrt{\lvert b \rvert} \bigr)},
\]
with $\tilde \Na$ and $\tilde \Nb$ depending only on the dimension. This way we recover the original result of \cite{NakicTTV}.
\begin{remark}[Relation between $\kappa$ and $G$]
 In Theorem~\ref{thm:result2} and~\ref{thm:result1}, the parameters $\kappa$ (decay of high energies) and $G$ (grid size) are subject to the relation $G / \kappa \leq 18 \euler \sqrt{d}$ or $G / \kappa < 18 \euler \sqrt{d}$, respectively.
 This is in accordance with the intuition of uncertainty principles: delocalization in momentum space (large $\kappa$) corresponds to localization in position space, i.e.\ a fine grid (small $G$) is required in order to obtain an estimate as in Ineq.~\eqref{eq:result1}.
 It also seems that the condition on $G$ and $\kappa$ appears naturally when using Carleman estimates to prove a scale-free quantitative unique continuation result as in Theorem~\ref{thm:result2} and \ref{thm:result1}. Indeed, a similar assumption is required in analogue results for solutions of variable coefficient second order elliptic operators with Lipschitz continuous coefficients, see \cite{BorisovTV-15c}. There, on a technical level the Lipschitz constant assumes the role of $1/\kappa$ from our setting and our condition turns into a smallness condition on the Lipschitz constant in the main result of \cite{BorisovTV-15c}.
 \par
 However, one could ask if a quantitative unique continuation principle as in Theorem~\ref{thm:result2} and Theorem~\ref{thm:result1} holds for every pair $(\kappa,G)$.
 An indication for this is Proposition~5.6 in \cite{RousseauL-12} where the following statement is proven in the special case $V \equiv 0$: \textit{Let $\omega \subset \Lambda_L$ be open and $\kappa > 0$. Then for all functions $u = \sum_{k \in \NN} \alpha_k \phi_k$ with $\lvert \alpha_k \rvert \leq \exp ( - \kappa \sqrt{E_k})$, $k \in \NN$, we have $u \equiv 0$ if $u|_{\omega} \equiv 0$.}
 Even though it would be possible without much effort to turn this qualitative into a quantitative statement of the form
 \[
  \lVert \phi \rVert_{\omega}^2 \geq C \lVert \phi \rVert_{\Lambda_L}^2,
 \]
 the method in \cite{RousseauL-12} does not provide any control over the constant $C$ in terms of $\delta$, $L$, and $\kappa$, which we study in this note.
 One possibility to treat arbitrary $\kappa$ and $G$ might be a so-called chaining argument, as used in \cite{DonnellyF-88,Kukavica-98,Bakri-13} in the context of quantitative uniqueness results and nodal sets for solutions of the Schr\"odinger equation.
 However, in order to obtain a strong dependence of $C_{\sfuc}$ on the parameters $\delta$ and $\lVert V \rVert_\infty$ as in Theorem~\ref{thm:result2} and \ref{thm:result1}, a direct adaptation of these chaining arguments to our setting might not be feasible.
 \end{remark}
\begin{remark}[Optimality]
\label{rem:optimality}
 As observed by Jerison and Lebeau in Proposition~14.9 of \cite{JerisonL-99}, the square root in the exponent of Ineq.~\eqref{eq:Assumption_B} and \eqref{eq:Assumption_A} is optimal.
 The exponent $2/3$ of $\lVert V \rVert_\infty$ in $C_\sfuc$, is known from Meshkov's example \cite{Meshkov-92} to be optimal in the case of eigenfunctions of Schr\"odinger operators with complex-valued potentials.
 It is an open question whether it can be improved for real-valued potentials.
 Another question is whether one can still expect results as in Theorem~\ref{thm:result2} and \ref{thm:result1} if the exponential functions in \eqref{eq:Assumption_B} and \eqref{eq:Assumption_A} are replaced by polynomials.
 We will show in Section~\ref{sec:optimality} that this is not the case. For this purpose, we show that every $\phi \in C_0^\infty(\Lambda_L)$ satisfies such a polynomial condition. Hence polynomial summability of the $\lvert \alpha_k \rvert^2$ does not imply such a quantitative unique continuation principle.
\end{remark}

%\newpage

\section{Proofs}
\label{sec:proofs}
\subsection{Ghost dimension and interpolation inequalities}

In this subsection we restate two interpolation inequalities from \cite{NakicTTV}, on which the proof of Theorem~\ref{thm:result1} relies. For more details we refer to \cite{NakicTTV}.
\par
Given a measurable and bounded $V:\RR^d \to \RR$ and $L \in \NN$ we define extensions of $V_L$ and of the eigenfunctions $\psi_k$ (defined on $\Lambda_L$) to a larger cube $\Lambda_{RL}$ where $R$ is the least odd integer larger than $18 \euler \sqrt{d} + 2$.
The type of the extension will depend on the boundary conditions, see\ \cite{NakicTTV}. In the case of
\begin{itemize}
 \item periodic boundary conditions we extend both $V$ and $\psi_k$ periodically.
 \item Dirichlet boundary conditions we extend $V$ iteratively by symmetric reflections with respect to the boundary of $\Lambda_L$, and $\psi_k$ by antisymmetric reflections.
 \item Neumann boundary conditions we extend both $V$ and $\psi_k$ iteratively by symmetric reflections with respect to the boundary of $\Lambda_L$.
\end{itemize}
We will use the same symbol for the extended $V_L$ and $\psi_k$. Note that $V_{L}: \Lambda_{RL} \to \RR$ takes values in $[-\lVert V \rVert_\infty, \lVert V \rVert_\infty]$, the extended $\psi_k$ are elements of $W^{2,2} (\Lambda_{R L})$ with corresponding boundary conditions, they satisfy the eigenvalue equation $\Delta \psi_k = (V_L - E_k) \psi_k$ on $\Lambda_{R L}$ and their orthogonality relations remain valid.
\par
For a measurable and bounded $V : \RR^d \to \RR$, $L \in \NN$ and $\phi \in L^2 (\Lambda_L)$ recall that $\alpha_k = \langle \psi_k , \phi \rangle$ whence $\phi = \sum_{k \in \NN} \alpha_k \psi_k$.
We set $\omega_k := \sqrt{\lvert E_k \rvert}$ and define for $n \in \NN$ the function $F_n : \Lambda_{R L} \times \RR \to \CC$ by
\begin{equation*} \label{eq:F}
 F_n (x , x_{d+1}) = \sum_{k =1}^n \alpha_k \psi_k (x) \funs_k( x_{d+1}) ,
\end{equation*}
where $s_k : \RR \to \RR$ is given by
\[
\funs_k(t)=\begin{cases}
	\sinh(\omega_k t)/\omega_k, & E_k>0,\\
	x, & E_k=0,\\
	\sin(\omega_k t)/\omega_k, & E_k<0.
\end{cases}
\]
Note that we suppress the dependence of $F_n$ on $V$, $L$ and $\phi$.
The function $F_n$ fulfills the handy relations
\begin{equation*} \label{eq:DeltaF}
\Delta F_n = \sum_{i=1}^{d+1} \partial^2_{i} F_n  =  V_L F_n \quad \text{on} \quad  \Lambda_{R L} \times \RR
\end{equation*}
and
\begin{equation*} \label{eq:F-phi}
\partial_{d+1} F_n (\cdot,0) = \sum_{k=1}^n \alpha_k \psi_k =: \phi_n \quad \text{on} \quad \Lambda_{R L}  .
\end{equation*}
In particular, we have $\lVert \partial_{d+1} F_n (\cdot,0) - \phi \rVert_{L^2 (\Lambda_{RL})} = \lVert \phi_n - \phi \rVert_{L^2 (\Lambda_{RL})} \to 0$ for $n \to \infty$.
\par

In the following, we recall that for $\Omega \subset \Lambda_{R L} \times \RR$
\[
 \lVert F_n \rVert_{H^1(\Omega)}^2
 =
 \lVert F_n \rVert_{L^2(\Omega)}^2
 +
 \sum_{i = 1}^{d+1} \lVert \partial_i F_n \rVert_{L^2(\Omega)}^2
\]
is the $1$-Sobolev norm.
Furthermore, in order to avoid confusion we now adapt the notation from \cite{NakicTTV} and define $\NN_{\mathrm{odd}} = \{1,3,5, \ldots\}$, $X_1 = \Lambda_L \times [-1,1]$, $R_3 = 9 \euler \sqrt{d}$, $\tilde X_{R_3} = \Lambda_{L + 2R_3} \times [-R_3 , R_3]$,
\begin{align*}
 S_1 &= \biggl\{x \in \RR^{d+1} \colon -x_{d+1} + \frac{x_{d+1}^2}{2} - \frac{\sum_{i=1}^d x_i^2}{4} > -\frac{\delta^2}{16} , x_{d+1}  \in [0,1] \biggr\} \subset \RR^{d+1}_+  , \\
 %\intertext{and}
 S_3 &= \biggl\{x \in \RR^{d+1} \colon -x_{d+1} + \frac{x_{d+1}^2}{2} - \frac{\sum_{i=1}^d x_i^2}{4} > -\frac{\delta^2}{4} , x_{d+1}  \in [0,1] \biggr\} \subset \RR^{d+1}_+ .
 \end{align*}
Moreover, for $L \in \NN $, a $(1,\delta)$-equidistributed sequence $Z$ and $i \in \{1,3\}$, we define the sets $U_i (L) = \cup_{j \in \ZZ^d \cap \Lambda_L} S_i (z_j)$. The following Propositions are variants of Propositions 3.4, 3.5 and 3.6 from \cite{NakicTTV}, see Remark~\ref{rem:Fn} below.

\begin{proposition} \label{prop:interpolation1}
For all $\delta \in (0,1/2)$, all $(1,\delta)$-equidistributed sequences $Z$, all measurable and bounded $V: \RR^d \to \RR$, all $L \in \NN_{\mathrm{odd}}$, all $n\in\NN$ and all $\phi \in L^2 (\Lambda_L)$ we have
\[
  \lVert  F_n \rVert_{H^1 (U_1 (L))} \leq D_1 \lVert (\partial_{d+1} F_n)_0 \rVert_{L^2 (W_\delta(L))}^{1/2} \lVert F_n \rVert_{H^1 (U_3 (L))}^{1/2} ,
\]
where
\begin{equation*} %\label{eq:C2}
   D_1^{-4} = \delta^{N_1 (1 + \lVert V \rVert_\infty^{2/3})}
\end{equation*}
and $N_1 = N_1 (d)$ is a constant depending on the dimension only.
\end{proposition}

\begin{proposition} \label{prop:interpolation2}
For all $\delta \in (0,1/2)$, all $(1,\delta)$-equidistributed sequences $Z$, all measurable and bounded $V: \RR^d \to \RR$, all $L \in \NN_{\mathrm{odd}}$, all $n \in \NN$ and all $\phi \in L^2 (\Lambda_L)$
we have
 \[
\lVert F_n \rVert_{H^1 (X_1)}
\leq D_2 \lVert F_n \rVert_{H^1 (U_1 (L))}^{\gamma} \lVert F_n \rVert_{H^1 (\tilde X_{\rout_3})}^{1- \gamma} ,
\]
where
\begin{equation*} %\label{eq:gamma}
\gamma = \left( \log_2 \left( \frac{6\euler\sqrt{d}}{\frac{1}{2} - \frac{1}{8}\sqrt{16-\delta^2}} \right) \right)^{-1} , \quad D_2^{-4/\gamma} = \delta^{N_2 (1+ \lVert V \rVert_\infty^{2/3})} ,
\end{equation*}
and $N_2 = N_2 (d)$ is a constant depending on the dimension only.
\end{proposition}
\begin{proposition} \label{prop:upper_lower}
For all $T > 0$, all measurable and bounded $V : \RR^d \to \RR$, all $L \in \NN_{\mathrm{odd}}$, all $n \in \NN$ and all $\phi \in L^2 (\Lambda_L)$ we have
\[
\frac{T}{2} \sum_{k=1}^n \lvert \alpha_k \rvert^2
\leq
\frac{\lVert F_n \rVert_{H^1 (\Lambda_{R L} \times [-T,T])}^2}{R^d}
\leq
2 T (1 + (1+\lVert V \rVert_\infty)T^2) \sum_{k=1}^n \beta_k(T)  \lvert \alpha_k \rvert^2 ,
\]
where
\[
 \beta_k(T) = \begin{cases}
	      1 & \text{if}\ E_k \leq 0,  \\
	      \mathrm{e}^{2T \sqrt{E_k} } & \text{if} \ E_k > 0 . \\
           \end{cases}
\]
\end{proposition}
\begin{remark} \label{rem:Fn}
The counterparts of Propositions \ref{prop:interpolation1}, \ref{prop:interpolation2} and \ref{prop:upper_lower} in \cite{NakicTTV} are formulated with $\phi \in \operatorname{Ran} \chi_{(-\infty , b]} (H_L)$ instead of $\phi \in L^2 (\Lambda_L)$, and
\begin{equation}\label{eq:Fb}
F^b(x,x_{d+1}) := \sum_{\genfrac{}{}{0pt}{2}{k \in \NN}{E_k \leq b}} \alpha_k \psi_k (x) \funs_k( x_{d+1}).
\end{equation}
instead of $F_n$.
However, the proofs in \cite{NakicTTV} do not depend on the particular choice of the index set $\{k \in \NN \colon E_k \leq b \}$ and apply to arbitrary finite index sets as well.
\end{remark}
\subsection{Proof of Theorem~\ref{thm:result1}}
First we consider the case $G = 1$, $\kappa \geq 18\euler\sqrt{d}$, and $L \in \NN_{\mathrm{odd}}$. We note that Proposition~\ref{prop:upper_lower} remains true if we replace $R$ by $1$, i.e.\ for all $T > 0$, $n \in \NN$ and $L \in \NN_{\mathrm{odd}}$ we have
\begin{equation}\label{eq:lower}
\frac{T}{2} \sum_{k=1}^n   \lvert \alpha_k \rvert^2 \leq \lVert F_n \rVert_{H^1 (\Lambda_{L} \times [-T,T])}^2
\leq 2 T (1 + (1+\lVert V \rVert_\infty)T^2) \sum_{k=1}^n  \beta_k(T)  \lvert \alpha_k \rvert^2 .
\end{equation}
We have $\tilde X_{\rout_3} \subset \Lambda_{R L} \times [-R_3 , R_3]$. By Ineq.~\eqref{eq:lower} and Proposition~\ref{prop:upper_lower} we have
\begin{align*}
\frac{\lVert F_n \rVert_{H^1 (\tilde X_{\rout_3})}^2 }{\lVert F_n \rVert_{H^1 (X_1)}^2 }
\leq
\frac{\lVert F_n \rVert_{H^1 (\Lambda_{R L} \times [-\rout_3, \rout_3])}^2 }{\lVert F_n \rVert_{H^1 (X_1)}^2}
\leq \frac{\sum_{k=1}^n \theta_k  \lvert \alpha_k \rvert^2}{\sum_{k=1}^n  \lvert \alpha_k \rvert^2} D_3^2
\end{align*}
where $D_3^2 = 4 R^d \rout_3 (1 + (1+\lVert V \rVert_\infty)\rout_3^2)$
and $\theta_k = \beta_k (\rout_3)$.
Now note that $\kappa \geq 18 \euler \sqrt{d} = 2R_3$. Therefore, Assumption \eqref{eq:Assumption_A} yields
\begin{align*}
\sum_{k\in\NN} \theta_k \lvert \alpha_k \rvert^2
&=
 \sum_{k\in \NN} \exp \left( 2R_3 \sqrt{\max\{0,E_k\}} \right) \lvert \alpha_k \rvert^2 \\
&\leq
 \sum_{k\in \NN} \exp \left( \kappa \sqrt{\max\{0,E_k\}} \right) \lvert \alpha_k \rvert^2
 \leq
 \Da \sum_{k \in \NN} \lvert \alpha_k \rvert^2 .
\end{align*}
% Therefore, we have $\sum_{k\in\NN} \theta_k \lvert \alpha_k \rvert^2 \leq \Da \sum_{k\in\NN} \lvert \alpha_k \rvert^2$.
Since
\[
n \mapsto \frac{\sum_{k=1}^n \theta_k \lvert \alpha_k \rvert^2 }{\sum_{k=1}^n \lvert \alpha_k \rvert^2 }
\]
is monotonously increasing, this implies
\[
\text{for all $n \in \NN$} \colon
\sum_{k=1}^n \theta_k \lvert \alpha_k \rvert^2  \leq \Da \sum_{k=1}^n \lvert \alpha_k \rvert^2 .
\]
Hence,
\[
 \frac{\lVert F_n \rVert_{H^1 (\tilde X_{\rout_3})}^2 }{\lVert F_n \rVert_{H^1 (X_1)}^2 } \leq \Da D_3^2
\]
We use Propositions~\ref{prop:interpolation1} and \ref{prop:interpolation2} and obtain
\begin{align*}
 \lVert F_n \rVert_{H^1 (\tilde X_{\rout_3})}
 &\leq \Da^{1/2} D_3 \lVert F_n \rVert_{H^1 (X_1)} \\
 &\leq \Da^{1/2} D_1^\gamma D_2 D_3
 \lVert F_n \rVert_{H^1 (\tilde X_{\rout_3})}^{1 - \gamma}
 \lVert (\partial_{x_{d+1}} F_n)_0 \rVert_{L^2 (W_\delta(L))}^{\gamma / 2}
 \lVert F_n \rVert_{H^1 (U_3(L))}^{\gamma / 2} .
\end{align*}
Since $U_3 (L) \subset \tilde X_{\rout_3}$ we have
\[
\lVert F_n \rVert_{H^1 (\tilde X_{\rout_3})} \leq
 \Da^{1/\gamma} D_1^2 D_2^{2/\gamma} D_3^{2/\gamma}
\lVert (\partial_{x_{d+1}} F_n)_0 \rVert_{L^2 (W_\delta(L))} .
\]
By Ineq.~\eqref{eq:lower}, the square of the left hand side is bounded from below by
\[
\lVert F_n \rVert_{H^1 (\tilde X_{\rout_3})}^2 \geq
\lVert F_n \rVert_{H^1 (\Lambda_L \times [-\rout_3 , \rout_3 ] ) }^2 \geq
 \frac{\rout_3}{2}\sum_{k=1}^n  \lvert \alpha_k \rvert^2  .
\]
Putting everything together we obtain by using $(\partial_{d+1} F_n)_0 = \phi_n$
\begin{equation*} \label{ucp1}
  \forall L \in \NN_{\mathrm{odd}}\colon \quad \tilde C_{\sfuc}^{\Ca} \lVert \phi_n \rVert_{L^2 (\Lambda_L)}^2 \leq  \lVert \phi_n \rVert_{L^2 (W_\delta(L))}^2
\end{equation*}
where $\tilde C_{\sfuc}^{\Ca} = \tilde C_{\sfuc}^{\Ca} (d,\delta , \Da , \lVert V \rVert_\infty) = (R_3/2) \Da^{-2/\gamma} D_1^{-4} ( D_2 D_3 )^{-4/\gamma}$.
For $D_1$, $D_2$, and $D_3$ we infer from \cite{NakicTTV}
\begin{equation*}
  D_1^{-4} \geq \delta^{-K_1 \lVert V \rVert_\infty^{2/3}}, \quad
  D_2^{-4/\gamma} \geq \delta^{K_2 \lVert V \rVert_\infty^{2/3}}, \quad \text{and} \quad
  D_3^{-4/\gamma} \geq \delta^{ K_3 \lVert V \rVert_\infty^{2/3}} .
\end{equation*}
and calculate
\begin{equation*}
 \Da^{-2/\gamma} = \left( \frac{\frac{1}{2} - \frac{1}{8} \sqrt{16 - \delta^2}}{6 \euler \sqrt{d}} \right)^{2\ln \Da / \ln 2}
 \geq
 \left( \frac{\delta^2 / 64}{6 \euler \sqrt{d}} \right)^{2\ln \Da / \ln 2}
 \geq \delta^{K_4 \ln \Da}
\end{equation*}
where $K_i$, $i \in \{1,2,\ldots , 4\}$, are constants depending only on the dimension. Hence,
\[
\tilde C_{\sfuc}^{\Ca} \geq \delta^{\tilde N \bigl(1 + \lVert V \rVert_\infty^{2/3} + \ln \Da \bigr)}
\]
with some constant $\tilde N = \tilde N (d)$. Letting $n$ tend to infinity and using $\lVert \phi_n - \phi \rVert_{L^2 (\Lambda_{L})} \to 0$ for $n \to \infty$, we conclude the statement of the theorem in the case $G = 1$, $\kappa \geq 18\euler\sqrt{d}$, and $L \in \NN_{\mathrm{odd}}$.
\par
Let now $\kappa > 0$ be arbitrary, $G \in (0,\kappa/(18\euler\sqrt{d})]$, and $L / G \in \NN_{\mathrm{odd}}$. We define the map $g : \Lambda_{L/G} \to \Lambda_L$, $g (y) = G \cdot y$. Then, on $\Lambda_{L/G}$ we have
\[
 -\Delta_{L/G} (\psi_k \circ g) + \left( G^2 V_L \circ g \right) \left( \psi_k \circ g \right) =  G^2 E_k   (\psi_k \circ g).
\]
Hence, the functions $\psi_k \circ g$ are an orthonormal basis of eigenfunctions of the operator $\tilde H_L = -\Delta_{L/G} + G^2 V_{L} \circ g$ with eigenvalues $\tilde E_k = G^2 E_k$. We apply our theorem with $G = 1$ and $\kappa / G$ to the function $\tilde \phi = \phi \circ g$ and obtain, using $\lVert \phi \rVert_{\Lambda_L}^2 = G^d \lVert \phi \circ g \rVert_{\Lambda_{L/G}}^2$,
\begin{equation} \label{eq:scaling}
  \lVert \phi \rVert_{W_\delta(L)}^2
  =
  G^{d} \lVert \phi \circ g \rVert_{W_\delta(L) / G}^2 
    \geq
  \left( \frac{\delta}{G} \right)^{\tilde N \bigl(1 + G^{4/3}\lVert V \rVert_\infty^{2/3} +  \ln \Da \bigr)} \lVert \phi \rVert_{\Lambda_{L}}^2 .
\end{equation}
The general case $L / G \in \NN$ follows by a similar scaling argument and the explicit dependence of $\tilde C_\sfuc^{\Ca}$ on the parameters, see \cite{NakicTTV} for details. \qed
\subsection{Proof of Theorem~\ref{thm:result2}}
\label{subsect:result2}
Recall that for a measurable and bounded $V : \RR^d \to \RR$ we denote by $V_L : \Lambda_L \to \RR$ its restriction to $\Lambda_L$, by $\Delta_L$ the Laplace operator on $L^2 (\Lambda_L)$ subject to either Dirichlet, Neumann or periodic boundary conditions, and by
\[
H_L = -\Delta_L + V_L \quad \text{on} \quad L^2 (\Lambda_L) .
\]
the corresponding Schr\"odinger operator. Moreover, we denote the eigenvalues of $H_L$ by $E_k$, $k \in \NN$, enumerated in increasing order and counting multiplicities, and fix a corresponding sequence $\psi_k$, $k \in \NN$, of normalized eigenfunctions.
\begin{lemma} \label{lemma:B->A}
 Let $C_1 , C_2 > 0$, $L \in \NN$, $V : \RR^d \to \RR$ measurable and bounded and $\phi \in L^2(\Lambda_L)$ satisfying
 \[
  \rVert \chi_{[\lambda, \infty)}(H_L) \phi \lVert^2_{\Lambda_L}
  \leq
  C_1
  \euler^{-(C_2 + \epsilon) \sqrt{\lambda + \lVert V_- \rVert_\infty}}
  \lVert \phi \rVert_{\Lambda_L}^2
  \quad
  \text{for every}\
  \lambda \in [ -\lVert V_- \rVert_\infty, \infty).
 \]
 Then we have
 \[
  \sum_{k = 1}^\infty \euler^{C_2 \sqrt{\max \{0 , E_k\}}} \lvert \alpha_k \rvert^2 \leq C_3 \sum_{k = 1}^\infty \lvert \alpha_k \rvert^2,
 \]
 where
 \[
 C_3 = \euler^{C_2(\pi + \lVert V_+ \rVert_\infty^{1/2})} \left( 1 + \frac{C_1 C_2 \pi}{1 - \euler^{- \epsilon \pi}} \right).
\]
\end{lemma}

For the proof of Lemma~\ref{lemma:B->A} and \ref{lem:result3} we shall need explicit formulas for the eigenvalues and eigenfunctions of the negative Laplacian $-\Delta_L$ on $L^2 (\Lambda_L)$. Depending on the boundary conditions we choose the index set $\index = \NN$ in the case of Dirichlet boundary conditions, $\index = \NN_0$ in the case of Neumann boundary conditions, and $\index = 2 \ZZ$ in the case of periodic boundary conditions. Then, the eigenvalues of $-\Delta_L$ are given by
\begin{equation} \label{eq:eigenvalues1}
 \lambda_y = \left( \frac{\pi}{L} \right)^2 \lvert y \rvert_2^2 , \quad y \in \index^d ,
\end{equation}
with corresponding normalized eigenfunctions
\begin{equation} \label{eq:eigenfunctions}
 e_y (x) =
 \begin{cases}
  \displaystyle\lVert e_y \rVert^{-1} \prod_{l=1}^d \sin \left( \frac{\pi y_l}{L} (x_l + L / 2) \right) & \text{in the case of Dirichlet b.c.},\\[1ex]
  \displaystyle\lVert e_y \rVert^{-1} \prod_{l=1}^d \cos \left( \frac{\pi y_l}{L} (x_l + L / 2) \right) & \text{in the case of Neumann b.c.}, \\[1ex]
  \displaystyle\lVert e_y \rVert^{-1} \exp \left( \frac{\mathrm{i} \pi}{L} y \cdot x  \right) & \text{in the case of periodic b.c.}.
 \end{cases}
\end{equation}
The normalization constants $\lVert e_y \rVert^{-1}$ can be easily calculated, though we will not need them.
Moreover, there exists a bijection $p : \NN \to \index^d$ such that
\[
 \lambda_{p (k)} , \quad k \in \NN ,
\]
is the $k$-th eigenvalue of $-\Delta_L$ enumerated in increasing order counting multiplicities. This bijection is unique up to permutations of sites $y \in \index^d$ with the same Euclidean norm.

\begin{proof}[Proof of Lemma~\ref{lemma:B->A}]
By the variational principle, we have for all $k \in \NN$
 \begin{equation} \label{eq:Ek}
\lambda_{p(k)} - \lVert V_- \rVert_\infty \leq E_k \leq \lambda_{p(k)} + \lVert V_+ \rVert_\infty .
 \end{equation}
%For $n \in \NN$ we denote by $\index_n^d = \{y \in \index^d : 1 \leq p^{-1} (y) \leq n\}$ the indices of the first $n$ eigenvalues of $\Delta_L$.
Using Ineq.~\eqref{eq:Ek} and $\sqrt{a+b} \leq \sqrt{a} + \sqrt{b}$ for $a,b\geq 0$, and Eq.~\eqref{eq:eigenvalues1} we obtain
\begin{align*}
 \sum_{k = 1}^\infty \euler^{C_2 \sqrt{\max\{0,E_k\}}} \lvert \alpha_k\rvert^2
 &\leq
 \euler^{C_2 \lVert V_+ \rVert_\infty^{1/2}} \sum_{k=1}^\infty  \euler^{C_2 \lvert p (k) \rvert \pi / L} \lvert \alpha_k \rvert^2  \\
 &\leq
 \euler^{C_2 \lVert V_+ \rVert_\infty^{1/2}} \sum_{l =1}^\infty
 \sum_{\genfrac{}{}{0pt}{1}{k\in\NN :}{l-1 \leq \lvert p (k) \rvert < l}}
 \!\!\!\!
 \euler^{C_2 l \pi / L} \lvert \alpha_{k} \rvert^2 =: \euler^{C_2 \lVert V_+ \rVert_\infty^{1/2}} S .
%  .
\end{align*}
By a telescoping argument we have
\begin{align*}
 S
 & =
 \euler^{C_2 \pi / L}\sum_{l =1}^\infty
 %\!\!\!\!
 \sum_{\genfrac{}{}{0pt}{1}{k\in\NN :}{l-1 \leq \lvert p (k) \rvert < l}}
 \!\!\!\!
  \lvert \alpha_{k} \rvert^2
 +
 \sum_{l =2}^\infty
  \left( \euler^{C_2 l \pi / L} - \euler^{C_2 (l - 1) \pi / L} \right) \sum_{m = l}^\infty
  %\!\!\!\!
 \sum_{\genfrac{}{}{0pt}{1}{k\in\NN :}{m-1 \leq \lvert p (k) \rvert < m}}
 \!\!\!\!\!\!\!\!   \lvert \alpha_{k} \rvert^2  \\
 & =
 \euler^{C_2 \pi / L}
 \sum_{k=1}^\infty
  \lvert \alpha_{k} \rvert^2
 +
 \sum_{l =2}^\infty
  \left( \euler^{C_2 l \pi / L} - \euler^{C_2 (l - 1) \pi / L} \right)
 \sum_{\genfrac{}{}{0pt}{1}{k\in\NN :}{l-1 \leq \lvert p (k) \rvert}}
  \lvert \alpha_{k} \rvert^2  .
\end{align*}
Since $\lvert y \rvert = \sqrt{\lambda_{y}} L / \pi$ and $E_k \geq \lambda_{p(k)} - \lVert V_- \rVert$ by Ineq.~\eqref{eq:Ek}, we have
\begin{align*}
 \sum_{\genfrac{}{}{0pt}{1}{k\in\NN :}{l-1 \leq \lvert p (k) \rvert}}
  \!\!\!\!\!\!  \lvert \alpha_{k} \rvert^2
  \,\,\,\,
 &=
 \!\!\!\!\!\!\!\!\!\!
 \sum_{\genfrac{}{}{0pt}{1}{k\in\NN:}{\lambda_{p(k)} \geq [(l - 1) \pi / L]^2}}
 \!\!\!\!\!\!\!\! \!\!\!\!\!\!
 \lvert \alpha_k \rvert^2
  \quad\,\,\leq
  \!\!\!\!\!\!\!\!\!\!\!
  \sum_{\genfrac{}{}{0pt}{1}{k\in\NN:}{E_k \geq [(l-1) \pi /L]^2 - \lVert V_- \rVert_\infty}}
  \!\!\!\!\!\!\!\!\!\!\!\!\!\!\!\!\!\!\!\!\!
  \lvert \alpha_k \rvert^2
  \,\,\,= \,\,\,
  \left\lVert \chi_{I_l} (H_L) \phi \right\rVert^2 ,
\end{align*}
where $I_l = [ ((l-1) \pi /L)^2 - \lVert V_- \rVert_\infty , \infty)$.
Finally, we use $\exp (C_2 l \pi / L) - \exp (C_2 (l - 1) \pi / L)  \leq (C_2 \pi/L) \exp (C_2 l \pi / L )$ and our assumption on the spectral projector to find
\begin{align*}
 \sum_{k \in \NN} \euler^{C_2 \sqrt{\max\{0,E_k\}}} \lvert \alpha_k\rvert^2
 &\leq \euler^{C_2 \lVert V_+ \rVert_\infty^{1/2}}
 \!\!\left(
 \euler^{C_2 \pi / L} \lVert \phi \rVert^2 + \frac{C_2 \pi}{L} \sum_{l=2}^\infty \euler^{C_2 l \pi / L} \left\lVert \chi_{I_l} (H_L) \phi \right\rVert^2
 \right) \\
 & \leq \euler^{C_2 \lVert V_+ \rVert_\infty^{1/2}}
 \left(
 \euler^{C_2 \pi / L} + \frac{C_1 C_2 \pi}{L} \euler^{C_2 \pi / L} \sum_{l=1}^\infty
  \euler^{- \epsilon l \pi / L}
 \right)
 \lVert \phi \rVert^2\\
 &\leq \euler^{C_2 (\pi + \lVert V_+ \rVert_\infty^{1/2})}
 \left(
  1 + \frac{C_1 C_2 \pi}{L ( 1 - \euler^{- \epsilon \pi / L})}
 \right)  \lVert \phi \rVert^2 .
\end{align*}
 Since the map $\NN \ni L \mapsto (L (1 - \euler^{- \epsilon \pi /L}))^{-1}$ is monotonously decreasing we finally obtain
 \[
\sum_{k \in \NN} \euler^{C_2 \sqrt{\max\{0,E_k\}}} \lvert \alpha_k\rvert^2  \leq  \euler^{C_2(\pi + \lVert V_+ \rVert_\infty^{1/2})} \left( 1 + \frac{C_1 C_2 \pi}{1 - \euler^{- \epsilon \pi}} \right) \lVert \phi \rVert^2.
  \qedhere
 \]
\end{proof}

Now we are in position to prove Theorem~\ref{thm:result2}.
\begin{proof}[Proof of Theorem~\ref{thm:result2}]
First we consider the case $\kappa > 18 e \sqrt{d} = 2R_3$ and $G = 1$.
Hence we have $\epsilon := \kappa  - 2R_3 > 0$.
Note that \eqref{eq:Assumption_B} implies
\begin{equation}\label{eq:Assumption_B_reformulated}
\left\lVert \chi_{[\lambda , \infty)} (H_L) \phi \right\rVert_{\Lambda_L}^2
 \leq
 \Db \euler^{- ( 2 R_3 + \epsilon) \sqrt{\lambda + \lVert V_- \rVert_\infty}} \lVert \phi \rVert_{\Lambda_L}^2.
\end{equation}
From Lemma~\ref{lemma:B->A} and \eqref{eq:Assumption_B_reformulated} we infer that Assumption~\eqref{eq:Assumption_A} of Theorem~\ref{thm:result1} is satisfied with
\begin{align*}
 \Da = \euler^{2R_3  (\pi + \lVert V_+ \rVert_\infty^{1/2})} \left( 1 + \frac{2 \Db R_3   \pi}{1 - \euler^{- \epsilon \pi}} \right) .
\end{align*}
Hence, we can apply Theorem~\ref{thm:result1} and obtain
\begin{equation*} \label{eq:after_thm1}
\lVert \phi \rVert_{W_\delta (L)}^2
\geq  \delta^{\Na \bigl(1 + \lVert V \rVert_\infty^{2/3} +  \ln \Da \bigr)} \lVert \phi \rVert_{L^2 (\Lambda_L)}^2 .
\end{equation*}
As an upper bound for $\ln \Da$ we use $x^{1/2} \leq 1 + x^{2/3}$, $R_3,\Db \geq 1$ and
\[
 - \ln (1 - \euler^{-\epsilon \pi})
 =
 \ln \left( 1 + \frac{\euler^{-\epsilon \pi}}{1 - \euler^{-\epsilon \pi}} \right)
 \leq
 \frac{\euler^{-\epsilon \pi}}{1 - \euler^{-\epsilon \pi}}
 =
 \frac{1}{\euler^{\epsilon \pi} - 1}
 \leq
 \frac{1}{\epsilon \pi} ,
 \]
and find
\begin{align*}
 \ln \Da
 & = 2R_3\pi + 2 R_3  \lVert V \rVert_\infty^{1/2} + \ln (1 - \euler^{-\epsilon \pi} + 2 \Db R_3  \pi) - \ln (1 - \euler^{-\epsilon \pi}) \\
 & \leq 2R_3 (\pi+1) + 2 R_3  \lVert V \rVert_\infty^{2/3} + \ln (\Db) + \ln (3R_3  \pi) + (\epsilon \pi)^{-1} .
\end{align*}
Hence there is a constant $\tilde N$ depending only on the dimension such that
\begin{align*}
 \ln \Da &\leq \tilde N \left(1+\lVert V \rVert_{\infty}^{2/3} + \epsilon^{-1} + \ln \Db \right) .
\end{align*}
This shows the statement of the theorem in the case $\kappa > 18\euler\sqrt{d}$ and $G = 1$. The general case follows by scaling, analogously to the end of the proof of Theorem~\ref{thm:result1}.
\end{proof}

\section{Discussion on optimality}\label{sec:optimality}

In Remark~\ref{rem:optimality} we discussed whether the class of functions $\phi$ satisfying
\begin{equation} \label{eq:polynomial}
 \sum_{k \in \NN} \max \{ 0, E_k \}^\kappa \lvert \alpha_k \rvert^2
 \leq
 D_B \sum_{k \in \NN} \lvert \alpha_k \rvert^2
\end{equation}
for some $D_B,\kappa > 0$ can still exhibit a unique continuation principle as in Theorem~\ref{thm:result1}.
The following lemma leads to a counterexample in the case $V \equiv 0$.

\begin{lemma}
\label{lem:result3}
 Let $L > 0$, $ V \equiv 0$, $\kappa > 0$, $\phi \in C_0^\infty(\Lambda_L)$.
 Then there is $C = C(\phi,L, \kappa) > 0$ such that
 \[
  \sum_{k \in \NN} \lvert E_k \rvert^\kappa \lvert \alpha_k \rvert^2 < C.
 \]
\end{lemma}
Now let $\phi \in C_0^\infty(\Lambda_L)$ be non-zero and vanishing on $W_\delta(L)$. By Lemma~\ref{lem:result3}, $\phi$ satisfies \eqref{eq:polynomial} with $\Da := C / \lVert \phi \rVert_{\Lambda_L}^2$, but not
$\lVert \phi \rVert_{W_\delta (L)}^2
\geq C_\sfuc \lVert \phi \rVert_{\Lambda_L}^2$. This shows the following corollary.
\begin{corollary}
 The statement of Theorem~\ref{thm:result1} with $\exp ( \kappa \sqrt{\max \allowbreak \{ 0, \allowbreak E_k \}} )$ replaced by $\max \{0, \allowbreak E_k \}^\kappa$ cannot hold.
\end{corollary}
\begin{proof}[Proof of Lemma~\ref{lem:result3}]
 Since the eigenfunctions and eigenvalues of $-\Delta$ on $\Lambda_L$ are explicitly known, cf. Section~\ref{subsect:result2}, we can replace the sum on the left hand side by
 \begin{align*}
 \sum_{k \in \NN} \lvert E_k \rvert^\kappa \lvert \alpha_k \rvert^2
 =
 \sum_{y \in \index^d} \left( \frac{\pi}{L} \right)^{2 \kappa} \lvert y \rvert_2^{2 \kappa} \lvert \langle e_y , \phi \rangle \rvert^2
  & \leq
  \left( \frac{\pi}{L} \right)^{2 \kappa} \sum_{y \in \index^d} \lvert y \rvert_2^{2 N} \lvert \langle e_y , \phi \rangle \rvert^2
 \end{align*}
 where $N \in 2\NN$ is the least even integer larger than $\kappa$.
 For the eigenfunctions, see Eq.~\eqref{eq:eigenfunctions}, we have $\partial_i^{N} e_y = -(\pi / L)^{N} \lvert y_i \rvert^{N} e_y$ for $i \in \{ 1, \cdots, d \}$. We calculate using integration by parts
 \begin{align*}
  &\sum_{y \in \index^d} \lvert y \rvert_2^{2 N} \lvert \langle e_y , \phi \rangle \rvert^2
  \leq
  N \sum_{i = 1}^d \sum_{y \in \index^d} \lvert y_i \rvert^{2 N} \lvert \langle e_y , \phi \rangle \rvert^2
  =
  N \left( \frac{L}{\pi} \right)^{2N} \sum_{i = 1}^d \sum_{y \in \index^d} \lvert \langle \partial_{i}^{N} e_y , \phi \rangle \rvert^2\\
  & \quad=
  N \left( \frac{L}{\pi} \right)^{2N} \sum_{i = 1}^d \sum_{y \in \index^d} \lvert \langle e_y , \partial_{i}^{N} \phi \rangle \rvert^2
  =
  N \left( \frac{L}{\pi} \right)^{2N} \sum_{i = 1}^d \lVert \partial_{i}^{N} \phi \rVert_{\Lambda_L}^2 . \qedhere
 \end{align*}
\end{proof}

\section*{Acknowledgments} We thank I.~Veseli\'c for the advice to pursue this research direction and for stimulating discussions during the completion of this work. We also gratefully acknowledge Thomas Kalmes and Ivica Naki\'c for helpful discussions.
The authors enjoyed hospitality at the University of Zagreb where parts of this work were done. This visit was supported by the bi-national German-Croatian DAAD-MZOS project \emph{The cost of controlling the heat flow in a multiscale setting}. Moreover, this work has been partially supported by the DFG grant Ve 253/6-1
\emph{Eindeutige-Fortsetzungsprinzipien und Gleichverteilungseigenschaften von Eigenfunktionen}.

\end{document}